\documentclass[11pt]{amsart}
\usepackage{latexsym, mathrsfs, color, tikz, multirow, bbm, mathtools}
\usepackage{graphics}
\usepackage{subcaption}
\input{xy}
\xyoption{all}

\usepackage[colorlinks=true, pdfstartview=FitV,
 linkcolor=black,citecolor=black,urlcolor=black]{hyperref}
 
\usepackage{tikz}

\usepackage{biblatex}
\usepackage{algorithm}
\usepackage{algpseudocode}

\usepackage{tikz-cd}
\usepackage{comment}

\numberwithin{equation}{section}
\theoremstyle{plain}
\newtheorem{Thm}[equation]{Theorem}
\newtheorem{Prop}[equation]{Proposition}
\newtheorem{Cor}[equation]{Corollary}
\newtheorem{Lem}[equation]{Lemma}

\theoremstyle{definition}
\newtheorem{Def}[equation]{Definition}
\newtheorem{Exa}[equation]{Example}
\newtheorem{Rmk}[equation]{Remark}

\newenvironment{red}{\relax\color{red}}{\relax}
\newenvironment{blue}{\relax\color{blue}}{\hspace*{.5ex}\relax}

\newcommand{\ZZ}{\mathbb{Z}}
\newcommand{\BB}{\mathcal{B}}
\newcommand{\PP}{\mathcal{P}}
\newcommand{\LL}{\mathcal{L}}

\newcommand{\ber}{\begin{red}}
\newcommand{\er}{\end{red}}
\newcommand{\beb}{\begin{blue}}
\newcommand{\eb}{\end{blue}}

\newcommand{\bw}{{\boldsymbol{w}}}
\newcommand{\bv}{{\boldsymbol{v}}}

\begin{document}
	
\title[Answering Two OPAC Problems Involving Banff Quivers]{Answering Two OPAC Problems Involving Banff Quivers}

\author[T. J. Ervin]{Tucker J. Ervin}
\address{Department of Mathematics, University of Alabama,
	Tuscaloosa, AL 35487, U.S.A.}
\email{tjervin@crimson.ua.edu}
\author[B. Jackson]{Blake Jackson}
\address{Department of Mathematics, University of Connecticut,
	Storrs, CT 06269, U.S.A.}
\email{blake.jackson@uconn.edu}
	
\begin{abstract}
    In a post on the Open Problems in Algebraic Combinatorics (OPAC) blog, E. Bucher and J. Machacek posed three open problems: OPAC-033, OPAC-034, and OPAC-035.
    These three problems deal with the relationships between three infinite classes of quivers: the Banff, Louise, and $\PP$ quivers.
    OPAC-034 asks whether or not every Banff quiver can be verified to be Banff by only considering sources and sinks, and OPAC-035 asks whether or not every Banff quiver is contained in the class $\PP$.
    We give an answer to both questions, showing that every Banff quiver can be verified to be Banff by using sources and sinks, and therefore that every Banff quiver lives in the class $\PP$.
    We also make some progress on OPAC-033, showing a result similar to our result OPAC-034 for Louise quivers.
\end{abstract}

\maketitle

\section{Introduction} \label{sec-introduction}
Open Problems in Algebraic Combinatorics (OPAC) is a conference that hosts academics from around the globe to discuss open problems at the forefront of research in algebraic combinatorics.
In 2019, the conference created a running document containing some current open problems as well as short introductions to each. 
Last updated in 2022, the document and the associated blog host 37 open problems in various areas of algebraic combinatorics. 
In January 2021, Bucher and Machacek submitted three new problems to the blog: OPAC-033, OPAC-034, and OPAC-035 \cite{bucher_banff_2021}. 
These three problems deal with combinatorial objects called quivers; specifically, these problems are about the relationships between the classes of Banff ($\BB$), Louise ($\LL)$, and $\PP$ quivers (as well as the classes $\BB'$ and $\PP'$ which are closely related to the Banff and $\PP$ classes).

The first of these classes of quivers to be introduced was the class $\mathcal{P}$ in a 2008 article by Kontsevich and Soibelman \cite{kontsevich_stability_2008}. For the authors of this paper, the class $\PP$ has the property that the corresponding 3-dimensional Calabi-Yau category is rigid. 
Banff quivers were introduced by Muller \cite{muller_locally_2013} in 2012 in order to study locally acyclic cluster algebras. 
The class of Banff quivers is a large class of quivers associated with locally acyclic cluster algebras.
Building on the work of Konsevich and Soibelman, Ladkani \cite{ladkani_cluster_2013} proposed a related class of quivers called $\PP'$ in 2013.
In 2016, Lam and Speyer \cite{lam_cohomology_2022} introduced a new class of quivers --- the Louise quivers --- while investigating the cohomology of cluster varieties for locally acyclic cluster algebras.
In 2018, Bucher and Machacek \cite{bucher_reddening_2020} showed that Banff quivers admit reddening sequences and introduced the $\BB'$ class of quivers.
They also raised some questions about the relationships between the classes of Banff, Louise, $\PP$, $\BB'$, and $\PP'$ quivers.
Eventually, the conjectures in these papers would be condensed into the OPAC blog post in January 2021. 

\textbf{OPAC-033:} \textit{Find an example of a Banff quiver that is not Louise or prove no such quiver exists.} 

\textbf{OPAC-034:} \textit{Find an example of a quiver in $\BB \setminus \BB'$  or prove no such quiver exists.}

\textbf{OPAC-035:} \textit{Find an example of a Banff quiver that is not in the class $\PP$ or prove no such quiver exists.}

This paper answers OPAC-034 and OPAC-035.
Specifically, we show that $\BB = \BB'$ and that, consequently, $\BB \subset \PP$.
Due to the work of Muller \cite[Lemma 8.13]{muller_skein_2016}, we know that the class $\BB'$ has the property that the upper cluster algebra equals the cluster algebra for quantum cluster algebras. 
Therefore, Muller's results extend to the entire class of Banff quivers.
We also define a new class of quiver $\LL'$ which is the counterpart of $\BB'$ for the class of Louise quivers and show that, by a similar argument to our main theorem, $\LL = \LL'$. 

The definitions of all of the classes of quivers can be found in Section~\ref{sec-background}.
The proofs of OPAC-034 and OPAC-035 and the partial result dealing with OPAC-033 are in Section~\ref{sec-proof}. 
Future work will attempt to make progress on OPAC-033.

\subsection{Acknowledgements}

We would like to thank Eric Bucher and John Machacek for their discussion and comments on an earlier draft.
We would also like to thank our advisor, Kyungyong Lee, for his guidance and support.

\section{Background}\label{sec-background}

All six of the classes listed in the introduction are infinite families of combinatorial objects known as \textit{quivers}.
Quivers were first introduced by Gabriel \cite{gabriel_unzerlegbare_1972}, but the definition of quivers has gone through many iterations since then.
Our definition of quiver is as follows.

\begin{Def}\label{def-quiver}
    A \textbf{quiver} $Q = (Q_0, Q_1)$ is a finite multidigraph without loops and directed 2-cycles where $Q_0$ is the set of vertices and $Q_1$ is the set of arrows.
    The elements of $Q_0$ are indexed by the numbers $1, 2, ... n$. 
    A quiver $Q$ is \textbf{acyclic} if it contains no directed cycles of any length.
    A vertex $i \in Q_0$ is a \textbf{source} (respectively \textbf{sink)} if there are no arrows in $Q_1$ which end (respectively begin) at $i$.
    If $C \subseteq Q_0$ is a subset of vertices of $Q$, then $Q_C = (C, Q_1')$ is the \textbf{(full) subquiver of} $Q$ \textbf{induced by} $C$ where $Q_1' \subseteq Q_1$ is the collection of arrows in $Q$ which start and end at vertices in $C$.
    When there is no risk of confusion, we abuse notation and call $C$ the subquiver induced by $C$.
\end{Def}

The most fundamental operation on these objects is called \textit{mutation}, which goes back at least as far as 1973 in the study of reflection functors for quiver representations \cite{bernstein_coxeter_1973}.

\begin{Def}\label{def-mutation}
    Let $i \in Q_0$ be a vertex in $Q$. 
    Then $\mu_i (Q)$ is the \textbf{mutation} of $Q$ at vertex $i$ and is the quiver obtained from $Q$ in the following way:
    \begin{enumerate}
        \item for each path $j \to i \to k$ in $Q$ of length 2 passing through $i$, add an arrow $j \to k$ in $\mu_i (Q)$
        \item reverse all arrows which begin or end at vertex $i$
        \item delete any 2-cycles that have appeared as a result of step 1.
    \end{enumerate}
    A quiver $Q'$ is \textbf{mutation-equivalent} to $Q$ if there is a finite sequence of vertices $\bw = [i_1, i_2, ..., i_k]$ in $Q_0$ such that $Q' = \mu_{i_k}(\mu_{i_{k-1}}(...(\mu_{i_1}(Q))...)) = \mu_\bw(Q)$.
    The sequence $\bw$ is called a \textbf{mutation sequence}.
    Any vertex $i$ that can be mutated at is called \textbf{mutable}, and in this paper all vertices are mutable.
\end{Def}

Mutation is an involutive operation that allows us to produce new quivers from existing quivers.
By starting with a quiver $Q$ (or a set of quivers), we can mutate (sometimes indefinitely) to produce a set of quivers that are mutation-equivalent to the starting set.
Mutating a general quiver $Q$ at an arbitrary vertex $i$ will often cause new arrows to be added (i.e. if there are any directed 2-paths through $i$, new arrows might be introduced after mutation at $i$).
However, if the chosen vertex $i$ is a source or a sink, then there can be no directed 2-paths through $i$.
In this case, we can be assured that mutation at $i$ will never introduce new arrows.
This motivates the following definitions and results.

\begin{Def} \label{def-source-sink-mutation}
    Let $Q$ be a quiver and $\bw = [w_1, w_2, \dots, w_k]$ a mutation sequence.
    We say that $\bw$ is a \textbf{source (sink) mutation-sequence} if $w_i$ is a source (sink) in $\mu_{\bw_i}(Q)$, where $\bw_i = [w_1, w_2, \dots, w_{i-1}]$.
\end{Def}

\begin{Def}{\cite{bang-jensen_classes_2018}} \label{def-acyclic-ordering}
    Let $Q$ be a quiver on $n$ vertices.
    An ordering $a_1 \prec a_2 \prec \dots \prec a_n$ on the vertices is an \textbf{acyclic ordering} if whenever there exists an arrow $a_i \to a_j$ in $Q_1$, then $a_i \prec a_j$.
\end{Def}

\begin{Prop}{\cite[Proposition 3.1.1]{bang-jensen_classes_2018}} \label{prop-acyclic-source-sink}
    Let $Q$ be an acyclic quiver.
    Then $Q$ has at least one source and at least one sink.
\end{Prop}

\begin{Prop}{\cite[Proposition 3.1.2]{bang-jensen_classes_2018}} \label{prop-acyclic-ordering}
    Let $Q$ be an acyclic quiver.
    Then there exists an acyclic ordering of its vertices.
\end{Prop}

The following is a well-known result, but we include it for completeness' sake.

\begin{Lem} \label{lem-acyclic-source-sink}
    Let $Q$ be an acyclic quiver on $n$ vertices.
    If $a_1 \prec a_2 \prec \dots \prec a_n$ is an acyclic ordering of the vertices of $Q$, then $\bw = [a_1, a_2, \dots, a_n]$ is a source mutation sequence and $\bv = [a_n, a_{n-1}, \dots, a_1]$ is a sink mutation sequence.
    Furthermore, if $\bw_i = [a_1, a_2, \dots, a_{i-1}]$ and $\bv_i = [a_n, a_{n-1}, \dots, a_{i+1}]$, then $\mu_{\bw_i}(Q)$ and $\mu_{\bv_i}(Q)$ are both acyclic for all $1 \leq i \leq n$.
\end{Lem}

\begin{proof}
    We prove only that $\bw$ is a source mutation sequence, as the argument for $\bv$ is similar.
    First, note that $\mu_{a_1}(Q)$ is acyclic with $a_2$ a source, as there is no $a_j$ such that $a_j \to a_2$ in $\mu_{a_1}(Q)$.
    Additionally, an acyclic ordering for $\mu_{a_1}(Q)$ is given by $a_2 \prec \dots \prec a_n \prec a_1$.
    We can then repeat the process for all of $\bw$, proving that $\mu_{\bw_i}(Q)$ is acyclic and $\bw$ is a source mutation sequence.
\end{proof}

Given the results listed previously and their relationship with sources and sinks, we now introduce another operation on a set of quivers--the \textit{triangular extension}.

\begin{Def}\label{def-triangular-extension}
    Let $A$ and $B$ be quivers. 
    A \textbf{triangular extension of $A$ and $B$} is a quiver $Q$ with $Q_0 = A_0 \cup B_0$ and $Q_1 = A_1 \cup B_1 \cup E$ where $E$ is a set of arrows satisfying exclusively one of
    $$i \to j \in E \implies i \in A_0, j \in B_0$$
    or 
    $$i \to j \in E \implies i \in B_0, j \in A_0,$$
    i.e., the arrows between $A$ and $B$ all point in the same direction in $Q$.
\end{Def}

If, for example, $A = \{a_1\}$ is the trivial quiver on one vertex and $B$ is any other quiver, then any non-trivial triangular extension $Q$ of $A$ and $B$ will have the vertex $a_1 \in Q_0$ as either a source or a sink of $Q$. 
Moreover, if $a_i \in A_0$ is a source (respectively sink) in $A$, then any triangular extension $Q$ of $A$ and $B$ with additional arrows $E$ from $A \to B$ (respectively $B \to A$) will have $a_i \in Q_0$ as a source (respectively sink).
These ideas will be at the core of our proofs in Section~\ref{sec-proof}.

The next definition is crucial to the definitions of Banff and Louise quivers.

\begin{Def}\label{def-biinfitnite}
    A \textbf{bi-infinite path} in $Q$ is a sequence $(i_a)_{a \in \ZZ}$ of mutable vertices such that $i_a \to i_{a+1}$ is an arrow in $Q$ for each $a \in \ZZ$. 
    A pair of vertices $(i,j)$ is a \textbf{covering pair} if $i \to j$ is an arrow in $Q$ which is not part of any bi-infinite path.
    Note that any arrow that is part of a directed cycle is always part of a bi-infinite path.
 \end{Def}

 We now provide some results on the nature of covering pairs which will be useful to us in Section~\ref{sec-proof}.

\begin{Prop}{\cite[Proposition 5.1]{muller_locally_2013}} \label{prop-acyclic-covering-pair}
    An arrow $i \to j$ is in a bi-infinite path if and only if there is a cycle along a directed path ending at $i$ and a cycle along a directed path beginning at $j$.
    Equivalently, an arrow $i \to j$ is a covering pair if and only if there is no cycle along a directed path ending at $i$ or there is no cycle along a directed path beginning at $j$.
\end{Prop}

\begin{Cor} \label{cor-acyclic-covering-pair}
    If $(i,j)$ is a covering pair, then the subquiver induced by the set of all vertices $A$ along a directed path ending at $i$, including $i$, or the subquiver induced by the set of all vertices $B$ along a direct path beginning at $j$, including $j$, are acyclic.
\end{Cor}

The following lemma comes from the proof of Theorem 3.2 in \cite{bucher_reddening_2020}.

\begin{Lem} \label{lem-covering-pair-triangulation}
    Let $(i,j)$ be a covering pair in a quiver $Q$.
    Additionally, let
    \begin{itemize}
        \item $A$ be the set of vertices $a$ such that there is an oriented path from $a$ to $i$, including $i$;

        \item $B$ be the set of vertices $b$ such that there is an oriented path from $j$ to $b$, including $j$;

        \item $C \supseteq B$ be the set of vertices such that $Q_0 = A \sqcup C$;

        \item $D \supseteq A$ be the set of vertices such that $Q_0 = B \sqcup D$.
    \end{itemize} 
    Then $Q$ is a triangular extension of $A$ and $C$ where every arrow between $A$ and $C$ points $A \to C$, and $Q$ is a triangular extension of $B$ and $D$ where every arrow between $B$ and $D$ points $D \to B$.
\end{Lem}

\begin{proof}
    Let $a \in A$, $b \in B$, $c \in C$, and $d \in D$.
    Then $c \to a$ would imply that there is a directed path from $c$ to $i$.
    Thus $c \in A$ is a contradiction.
    Hence, either $a \to c$ or $a$ and $c$ are not neighbors.
    A similar argument shows either $d \to b$ or $d$ and $b$ are not neighbors.
    Therefore, the quiver $Q$ is a triangular extension of $A$ and $C$ where every arrow between $A$ and $C$ points from $A \to C$, and $Q$ is a triangular extension of $B$ and $D$ where every arrow between $B$ and $D$ points $D \to B$.
\end{proof}

Finally, we define the classes of quivers that this paper is concerned with.
Each of the classes of Banff, $\BB'$, Louise, and $\LL'$ quivers are defined recursively; a quiver $Q$ is in a class, say $\mathcal{T}$, if specific subquivers of $Q$ are also in the class $\mathcal{T}$.
As far as we know, there are no ``constructive'' definitions of the classes of Banff, $\BB'$, Louise, or $\LL'$ quivers. 
An example of a constructive definition of $\PP'$ is given below in Example~\ref{exa-constructive}.
The existence of a constructive definition of these classes would almost certainly give an answer to OPAC-033 as well as alternative proofs to OPAC-034 and OPAC-035.

\begin{Def}\label{def-banff}
    The class $\BB$ of \textbf{Banff quivers} is the smallest class of quivers such that 
    \begin{itemize}
        \item any acyclic quiver is Banff,
        \item any quiver mutation equivalent to a Banff quiver is Banff, 
        \item and any quiver $Q$ with a covering pair $(i,j)$ where both $Q\setminus \{i\}$ and $Q\setminus \{j\}$ are Banff is a Banff quiver.
    \end{itemize}
 \end{Def}

 \begin{Def}\label{def-banffprime}
    The class $\BB'$ of \textbf{Banff prime quivers} is the smallest class of quivers such that
    \begin{itemize}
        \item any quiver without arrows is in $\BB'$,
        \item any quiver mutation equivalent to a quiver in $\BB'$ is in $\BB'$, 
        \item and any quiver $Q$ with an arrow $i \to j$ where $i$ is a source or $j$ is a sink, and both $Q\setminus \{i\}$ and $Q\setminus \{j\}$ are in $\BB'$ is in $\BB'$.
    \end{itemize}
 \end{Def}

 \begin{Def}\label{def-louise}
    The class $\LL$ of \textbf{Louise quivers} is the smallest class of quivers such that 
    \begin{itemize}
        \item any acyclic quiver is Louise,
        \item any quiver mutation equivalent to a Louise quiver is Louise, 
        \item and any quiver $Q$ with a covering pair $(i,j)$ where $Q\setminus \{i\}$, $Q\setminus \{j\}$, and $Q\setminus \{i,j\}$ are Louise is a Louise quiver.
    \end{itemize}
 \end{Def}

 \begin{Def}\label{def-louiseprime}
    The class $\LL'$ of \textbf{Louise prime quivers} is the smallest class of quivers such that 
    \begin{itemize}
        \item any quiver without arrows is in $\LL'$,
        \item any quiver mutation equivalent to a quiver in $\LL'$ is in $\LL'$, 
        \item and any quiver $Q$ with an arrow $i \to j$ where $i$ is a source or $j$ is a sink, and each of $Q\setminus \{i\}$, $Q\setminus \{j\}$, and $Q\setminus \{i,j\}$ are in $\LL'$ is in $\LL'$.
    \end{itemize}
 \end{Def}

 \begin{Def}\label{def-P}
    The class $\PP$ is the smallest class of quivers such that 
    \begin{itemize}
        \item the trivial quiver with one vertex and no arrows is in $\PP$,
        \item any quiver mutation equivalent to a quiver in $\PP$ is in $\PP$, 
        \item and any triangular extension of two quivers in $\PP$ is in $\PP$.
    \end{itemize}
 \end{Def}

 \begin{Def}\label{def-Pprime}
    The class $\PP'$ is the smallest class of quivers such that 
    \begin{itemize}
        \item the trivial quiver with one vertex and no arrows is in $\PP'$,
        \item any quiver mutation equivalent to a quiver in $\PP'$ is in $\PP'$, 
        \item and any triangular extension of a quiver in $\PP'$ with the trivial quiver is in $\PP'$.
    \end{itemize}
 \end{Def}

\begin{Exa}\label{exa-constructive}
    This is an example of an alternative and ``constructive'' definition for the class of $\PP'$ quivers. This gives a gradation of this class by the number of vertices of each quiver:
    \begin{enumerate}
        \item Start with the trivial quiver on one vertex, this is the only quiver with one vertex in $\PP'$
        \item To form the $(n+1)^{st}$ level of $\PP'$, take the entirety of the $n^{th}$ level, extended triangularly with the one vertex quiver in all possible ways, and find all quivers mutation-equivalent to the resulting quivers.
    \end{enumerate}
    Finding constructive definitions for Banff, $\BB'$, Louise, and $\LL'$ would almost surely give a proof or counter-example to OPAC-033.
\end{Exa}

\begin{Rmk} \label{rmk-subclasses}
    We naturally get that $\LL' \subseteq \LL$, $\BB' \subseteq \BB$, and $\LL \subseteq \BB$.
    However, there is strict containment $\PP' \subset \PP$ \cite[Remark 4.21]{ladkani_cluster_2013}.
\end{Rmk}

Finally, we state a proposition of Bucher and Machacek that allows us to prove OPAC-035 as a corollary of our main theorem.

\begin{Prop}{\cite[Proposition 3.4]{bucher_reddening_2020}} \label{prop-banff-p-containment}
    If a quiver is in $\BB'$, then it is in $\PP'$.
\end{Prop}

\section{Answering Two Problems And A Louise Lemma}\label{sec-proof}

\begin{Lem} \label{lem-triangular-extension-source-sink}
    Let $Q$ be a triangular extension of two quivers $A$ and $B$ such that every arrow between $A$ and $B$ points $A \to B$.
    If $\bw = [w_1, w_2, \dots, w_p]$ is a source mutation sequence of distinct vertices of $A$, then $\bw$ is a source mutation sequence of $Q$.
    Similarly, if $\bv = [v_1, v_2, \dots, v_q]$ is a sink mutation sequence of distinct vertices of $B$, then $\bv$ is a sink mutation sequence of $Q$
\end{Lem}

\begin{proof}
    We prove only the case for the source mutation $\bw$, as the proof for the sink mutation $\bv$ is analogous. 
    Suppose that $w_i$ is a source in $\mu_{\bw_i}(Q)$ for $\bw_i = [w_1, w_2, \dots, w_{i-1}]$ and $\mu_{\bw_i}(Q) \setminus \{w_1,w_2,\dots, w_{i-1}\}$ is a triangular extension of $A \setminus \{w_1,w_2,\dots, w_{i-1}\}$ and $B$, where every arrow points from $A \setminus \{w_1,w_2,\dots, w_{i-1}\}$ to $B$.
    Then $w_{i+1}$ is a source in $\mu_{\bw_{i+1}}(Q)$, as mutation at $w_i$ does not affect any arrows between $A \setminus \{w_1,w_2,\dots, w_{i-1}, w_i\}$ and $B$.
    Additionally, the quiver $\mu_{\bw_{i+1}}(Q) \setminus \{w_1,w_2,\dots, w_{i-1}, w_i\}$ is triangular extension of $A \setminus \{w_1,w_2,\dots, w_{i-1}, w_i\}$ to $B$, where every arrow points from $A \setminus \{w_1,w_2,\dots, w_{i-1}, w_i\}$ to $B$.
    As $w_1$ must be a source in $Q$, induction proves that $\bw = [w_1, w_2, \dots, w_p]$ is a source mutation sequence of $Q$ as long as it is a source mutation sequence of distinct vertices of $A$.
\end{proof}

\begin{Lem} \label{lem-covering-pair-mutation}
    Let $Q$ be a quiver with a covering pair $(i,j)$.
    Then there exists a source or sink mutation sequence $\bw$ avoiding $i$ and $j$ such that $P = \mu_\bw(Q)$ has $i$ as a source or $j$ as a sink, respectively.
    In particular, this forces $Q \setminus \{i\}$, $Q \setminus \{j\}$, and $Q \setminus \{i,j\}$ to be respectively mutation-equivalent to $P \setminus \{i\}$, $P \setminus \{j\}$, and $P \setminus \{i,j\}$.
\end{Lem}

\begin{proof}
    If $i$ is a source or $j$ a sink in $Q$, then setting $\bw = []$ gives the desired result.
    We may assume that $i$ is not a source and $j$ is not a sink.
    As $(i,j)$ is a covering pair in $Q$, we know that either there is no cycle along a directed path ending at $i$ or there is no cycle along a directed path beginning at $j$.

    \underline{Case 1: no cycle along a directed path ending at $i$}

    Let $A$ be the set of all vertices $a$ such that there exists a directed path from $a$ to $i$, including $i$ itself.
    Then $A$ is acyclic by Corollary~\ref{cor-acyclic-covering-pair}.
    Let $C$ be the complement of $A$.
    Then $Q$ is a triangular extension of $A$ and $C$ where every arrow between $A$ and $C$ points $A \to C$ by Lemma~\ref{lem-covering-pair-triangulation}.

    As $A$ is acyclic, there exists an acyclic ordering $a_1 \prec a_2 \prec \dots \prec a_m$ of $A$, where $m = |A|$, by Proposition~\ref{prop-acyclic-ordering}.
    Then $\bv = [a_1,a_{2},\dots,a_m]$ is a source mutation sequence of $A$ by Lemma~\ref{lem-acyclic-source-sink}.
    Thus $\bv$ is a source mutation sequence of $Q$ by Lemma~\ref{lem-triangular-extension-source-sink}.
    If $i = a_q$, then $i$ must be a source in $P = \mu_{\bw}(Q)$ for $\bw = [a_1,a_2,\dots,a_{q-1}]$.
    Additionally, as $j \notin A$ and $i \notin \bw$, we know that we have not mutated at $i$ or $j$.
    Thus $Q \setminus \{i\}$, $Q \setminus \{j\}$, and $Q \setminus \{i,j\}$ after mutation along $\bw$ respectively equal $P \setminus \{i\}$, $P \setminus \{j\}$, and $P \setminus \{i,j\}$.
    
    \underline{Case 2: no cycle along a directed path beginning at $j$}

    Repeat the above argument using sinks and the sets $B$ and $D$ described in Lemma~\ref{lem-covering-pair-triangulation}.
    The result then holds for both cases.
\end{proof}

We now arrive at our main result: a proof of OPAC-034 which shows that the Banff quivers and the $\BB'$ quivers are the same class.

\begin{Thm}[OPAC-034] \label{thm-opac-034}
    Every Banff quiver is mutation-equivalent to a quiver in $\BB'$, i.e., $\BB = \BB'$.
\end{Thm}

\begin{proof}
    We argue by induction on the number of vertices in a quiver.
    The single vertex quiver is acyclic, and thus it is in both $\BB$ and $\BB'$, completing the base case.
    Suppose then that every Banff quiver on $n$ vertices is contained in $\BB'$.
    Let $Q$ be a Banff quiver on $n+1$ vertices.
    As both $\BB$ and $\BB'$ are closed under mutation, we can assume that $Q$ has a covering pair $(i,j)$ such that $Q \setminus \{i\}$ and $Q \setminus \{j\}$ are Banff.
    If either $i$ is a source or $j$ is a sink, then $Q \in \BB'$ by the inductive hypothesis.
    We can then assume that $i$ is not a source and $j$ is not a sink and split into the two cases given by Proposition~\ref{prop-acyclic-covering-pair}.

    Then Lemma~\ref{lem-covering-pair-mutation} gives us a source or sink mutation-sequence $\bw$ such that either $i$ is a source or $j$ is a sink in $P = \mu_\bw(Q)$.
    Additionally, we get that $P \setminus \{i\}$ and $P \setminus \{j\}$ are mutation-equivalent to $Q \setminus \{i\}$ and $Q \setminus \{j\}$.
    Both $P \setminus \{i\}$ and $P \setminus \{j\}$ are then Banff and in $\BB'$ by the inductive hypothesis.
    This implies that $P \in \BB'$, forcing $Q \in \BB'$.
    As $Q$ was an arbitrary member of $\BB$, we have $\BB = \BB'$ by induction.
\end{proof}

\begin{Cor}[OPAC-035] \label{cor-opac-035}
    Every Banff quiver is contained in class $\PP$.
\end{Cor}

\begin{proof}
    As $\BB' \subseteq \PP'$ from Proposition~\ref{prop-banff-p-containment},  we have
    $$\BB = \BB' \subseteq \PP' \subset \PP$$
    from Theorem~\ref{thm-opac-034}.
\end{proof}

Finally, we make partial progress on OPAC-033 by proving a similar result to Theorem~\ref{thm-opac-034} for Louise quivers.

\begin{Cor} \label{cor-louise-louise-prime}
    Every Louise quiver is mutation-equivalent to a quiver in $\LL'$, i.e., $\LL= \LL'$.
\end{Cor}

\begin{proof}
    Repeat the argument of Theorem~\ref{thm-opac-034} with a Louise quiver $Q$ that has a covering pair $(i,j)$ such that $Q \setminus \{i\}$, $Q \setminus \{j\}$, and $Q \setminus \{i,j\}$ are Louise.
    Then Lemma~\ref{lem-covering-pair-mutation} gives us a source or sink mutation-sequence $\bw$ such that either $i$ is a source or $j$ is a sink in $P = \mu_\bw(Q)$.
    Hence $Q \setminus \{i\}$, $Q \setminus \{j\}$, and $Q \setminus \{i,j\}$ are respectively mutation-equivalent to $P \setminus \{i\}$, $P \setminus \{j\}$, and $P \setminus \{i,j\}$.
    This forces $Q \in \LL'$, proving that $\LL = \LL'$ by induction.
\end{proof}

\begin{Cor} \label{cor-subclasses-all}
    $\LL' = \LL \subseteq \BB = \BB' \subseteq \PP' \subset \PP$
\end{Cor}

The consequence of Theorem~\ref{thm-opac-034} and Corollary~\ref{cor-louise-louise-prime} in light of the remaining open problem OPAC-033 is that we only need to consider source/sink covering pairs. 
Due to Lemma~\ref{lem-covering-pair-mutation}, if there exists a covering pair $(i,j)$ in a quiver $Q$ such that $Q \setminus \{i\}$ and $Q \setminus \{j\}$ are Louise but $Q \setminus \{i,j\}$ is not Banff, then we can assume that either $i$ is a source or $j$ is a sink.

A possible approach to solving OPAC-033 would be to give constructive definitions for the Banff and Louise quiver classes in the style of Example~\ref{exa-constructive}. 
While this would make the proof of OPAC-033 almost trivial (as mentioned previously), we currently have little evidence that this is a feasible approach.

\begin{Rmk}
Note that being a member of any of the four classes $\BB$, $\LL$, $\PP'$, or $\PP$ is not a hereditary property (a property preserved by restricting to any full subquiver).
For example, take the quiver given by 
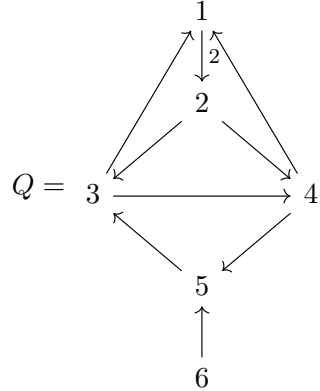
\[Q =\begin{tikzcd}
& 1 \\
& 2 \\
3 & & 4 \\
& 5\\
& 6
\arrow[from=1-2, to=2-2, "2"]
\arrow[from=2-2, to=3-1]
\arrow[from=2-2, to=3-3]
\arrow[from=3-1, to=1-2]
\arrow[from=3-1, to=3-3]
\arrow[from=3-3, to=1-2]
\arrow[from=3-3, to=4-2]
\arrow[from=4-2, to=3-1]
\arrow[from=5-2, to=4-2]
\end{tikzcd}\]
This quiver is Louise \cite[Remark 11.15]{muller_locally_2013}.
However, the subquiver on the vertices 1,2,3, and 4 is not mutation-equivalent to any triangular extension, since its mutation class consists entirely of isomorphic copies of itself.
As such, it is a counterexample of any of the four classes being hereditary.  
It also gives an example of a covering pair that fails for testing both Banff and Louise, as (6,5) is a covering pair with $Q \setminus \{5\}$ not in any of the four classes.
This gives a way to show $\BB \neq \LL$ if $\BB = \LL$ implies that membership in either class is a hereditary property.
\end{Rmk}

\begin{Prop}\label{prop-size-counterexample}
Let $Q$ be a Banff quiver on 5 or fewer vertices.
Then $Q$ is also Louise.
Consequently, any counterexample to $\BB = \LL$ must have 6 or more vertices.
\end{Prop}

\begin{proof}
    All quivers on 3 or fewer vertices that are Banff are also trivially Louise.
    We may then assume that $Q$ has 4 or 5 vertices.
    This forces $Q \setminus\{i,j\}$ to be mutation-acyclic for any pair of vertices $\{i,j\}$, as $Q \setminus\{i,j\}$ must admit a reddening sequence and has 3 or fewer vertices.
    Hence, the subquiver $Q \setminus \{i,j\}$ is both Banff and Louise. 
    Let $(i,j)$ be any covering pair such that $Q \setminus \{i\}$ and $Q \setminus \{j\}$ are both Banff.
    If $Q$ has 4 vertices, then $Q \setminus \{i\}$ and $Q \setminus \{j\}$ are Banff quivers on 3 vertices, making them both Louise. 
    Thus $Q$ is Louise.
    If $Q$ has 5 vertices, then $Q \setminus \{i\}$ and $Q \setminus \{j\}$ are Banff quivers on 4 vertices, which we have just shown to be Louise.
    Again, this proves that $Q$ is Louise.
    
    However, this argument cannot be extended to quivers on 6 vertices, as $Q \setminus \{i,j\}$ is a quiver on 4 vertices and is no longer guaranteed to be mutation-acyclic.
    Thus, the quiver $Q \setminus \{i,j\}$ is not guaranteed to belong to $\BB$, $\LL$, or $\PP$.
\end{proof}

\printbibliography

\end{document}